\RequirePackage{etex}
\documentclass[11pt, reqno]{amsproc}

\usepackage{amsmath,amssymb}
\usepackage{mathrsfs}
\usepackage{comment}

\usepackage[english]{babel}
\usepackage[utf8x]{inputenc}
\usepackage[T1]{fontenc}

\usepackage{tikz-cd}

\usepackage[top=3cm,bottom=2cm,left=3cm,right=3cm,marginparwidth=1.75cm]{geometry}
\usepackage{graphicx}
\usepackage[colorinlistoftodos]{todonotes}
\usepackage[colorlinks=true, allcolors=blue]{hyperref}
\usepackage{cite}
\usepackage{enumitem}
\setlist[1]{itemsep=3pt}
\usepackage[all,cmtip]{xy}

\usepackage[capitalise]{cleveref}
\usepackage{autonum}
\usepackage{setspace}

\newcommand{\beq}{\begin{equation}}
\newcommand{\eeq}{\end{equation}}

\usepackage[colorinlistoftodos]{todonotes}

\theoremstyle{plain}
\newtheorem{teorema}{teorema}[section]
\newtheorem{theorem}[teorema]{Theorem}
\newtheorem*{theorem*}{Theorem}

\newtheorem{lemma}[teorema]{Lemma}
\newtheorem{proposition}[teorema]{Proposition}

\theoremstyle{definition}
\newtheorem{definition}[teorema]{Definition}
\newtheorem{remark}[teorema]{Remark}
\newtheorem{notations}[teorema]{Notations}
\newtheorem{remarks}[teorema]{Remark}
\newtheorem{esempio}{esempio}[section]
\newtheorem{example}[esempio]{Example}

\numberwithin{equation}{section}

\let\oldtocsection=\tocsection
\let\oldtocsubsection=\tocsubsection
\let\oldtocsubsubsection=\tocsubsubsection
\renewcommand{\tocsection}[2]{\hspace{-1.2em}\oldtocsection{#1}{#2}}
\renewcommand{\tocsubsection}[2]{\hspace{-.2em}\oldtocsubsection{#1}{#2}}
\renewcommand{\tocsubsubsection}[2]{\hspace{0.8em}\oldtocsubsubsection{#1}{#2}}

\DeclareRobustCommand{\gobblefive}[5]{}

\makeatletter
\renewcommand\subsubsection{\@startsection{subsubsection}{3}%
  \z@{.5\linespacing\@plus.7\linespacing}{-.5em}%
  {\normalfont\bfseries}}
\makeatother

\newcommand{\R}{\mathbb{R}}

\newcommand{\dd}{\mathrm{d}}

\newcommand{\rest}[1]{\big\rvert_{#1}} 

\title[]{A proof of the Hamiltonian Thom Isotopy Lemma}
\author{Paolo Antonini}
\address{Dipartimento di Matematica e Fisica "E. De Giorgi", Universit\`a del Salento, Lecce (Italy)
}
\email{paolo.antonini@unisalento.it}
\author{Fabio Cavalletti} \address{Mathematics Area, SISSA, Trieste (Italy)}
\email{cavallet@sissa.it}
\author{Antonio Lerario}
\address{Mathematics Area, SISSA, Trieste (Italy)}
\email{lerario@sissa.it}

\begin{document}
\maketitle
\begin{abstract}
In this note we present a complete proof of the fact that all the submanifolds of a one parameter family of compact symplectic submanifolds inside a compact symplectic manifold are Hamiltonian isotopic. \end{abstract}

\setcounter{tocdepth}{1}

\smallskip

 \section{Introduction}
 In this paper we give a detailed proof of the following statement that is known as a folklore within the symplectic community.
  
 \begin{theorem}\label{symsubmanifolds}
	Let $(M,\omega)$ be a compact symplectic manifold with a smooth family $\{S_t\}_{t\in [0,1]}$ of closed symplectic submanifolds. 
	Then there is a smooth Hamiltonian isotopy $(\rho_t)_{t\in[0,1]}$ such that $\rho_t(S_0)=S_t$, $t\in[0,1]$.
\end{theorem}

This result is part of the lore of the symplectic community, 
with partial proofs scattered in some places, see for instance \cite{zbMATH05286973,zbMATH01123717, sieberttian}.  More details on the bibliography will be given in  Section \ref{HamThom}. 
The scope of this short note is to present a complete proof, filling eventually all the details.

\noindent{\textbf{Acknowledgements}}
It is a pleasure to thank N. Ciccoli, E. Meinranken and G. Dimitroglou Rizell
 for some useful discussions.  \\

\begin{notations}\nonumber
Below are some notations that we are going to use.

We shall consider time dependent objects (vector fields, differential forms, etc.)
that depend on a second parameter usually called $s\in [0,1]$. 
For instance, 
a time dependent vector field, depending on $s$ is a section $\zeta \in \Gamma([0,1]\times [0,1]\times M,q^*(TM))$, where $q:[0,1]\times [0,1]\times M \to M$ is the projection. We also write this section as
$\zeta:[0,1]_t \times [0,1]_s \times M \longrightarrow TM$ and put  ${\zeta}_t^s=\zeta(t,s)\in \Gamma(M,TM)$ for its values.
 Similar notations will be adopted for other objects as differential forms.

Given a manifold $X$, a submanifold $Y$ (or a subset), a fiber bundle $E$ over $X$ and a map $f:E\to Z$ (e.g. a form or a differential), the symbol $f\restriction Y$ will be used to denote the restriction of $f $ to $E|_Y$. The symbol $\rest{\,}$ is used to denote the classical restriction to subsets.
\end{notations}

\section{Preliminaries: vector fields, isotopies and Thom Isotopy Lemma}
Let $M$ be a smooth manifold; an (ambient) isotopy is a smooth map $\rho:M\times [0,1] \longrightarrow M$ such that $\rho_0=\operatorname{Id}_M$ and $\rho_t$ 
is a diffeomorphism for every $t\in [0,1]$. 
An isotopy is a curve starting from the identity inside $\operatorname{Diff}(M)$.
Differentiating it we get a time dependent vector field
 ${\zeta}_t$ satisfying:
\begin{equation}\label{isotopy}
\dfrac{d}{dt}\rho_t ={\zeta}_t \circ \rho_t, \quad t\in[0,1].\end{equation}
Viceversa, given a time dependent vector field $\zeta_t$ with $t\in [0,1]$, it is possible, under the usual assumptions on the existence of flows, to find a unique isotopy $\rho_t$ satisfying equation \eqref{isotopy}. This is called the {\em{flow}} \footnote{ somewhere called the reduced flow of $\zeta$ to distinguish with the full flow as a time dependent vector field}  of $\zeta_t$. 
If $\zeta$ is compactly supported (i.e. $\zeta_t$ is compactly supported for every $t\in [0,1]$) then the flow exists.

It is worth to give few details on the construction of $\rho_t$ given $\zeta_t$ as it will be used later to construct the Hamiltonian isotopy.
We assume $\zeta_t$, $t\in [0,1]$ is compactly supported and put $$\widehat{M}:= M\times [0,1]\subset M\times \mathbb{R}.$$ This is an auxiliary submanifold with boundary; $\widehat{M}$ will be also used next.
We have projections $\pi:M \times \mathbb{R} \to \mathbb{R}$ and $\pi_M:M \times \mathbb{R}\to M$ and embeddings $e_t:M \to M\times \mathbb{R}$ defined by  
 $e_t (x)=(x,t)$.

Now we
 extend $\zeta_t$
(keeping the same name) to a vector field $\zeta_t \in \Gamma(M \times \mathbb{R},\pi^*(TM))$ with the property that 
there is a fixed $\varepsilon >0$ such that 
$\zeta_t =0$ for $t \notin [-\varepsilon,1+\varepsilon]$. Then we convert it into the 
standard (autonomous) vector field
$$\widehat{\zeta}(x,t):=\zeta_t(x) +\partial_t \in \Gamma(M \times \mathbb{R},T(M\times \mathbb{R})).$$ 
It is clear that $\widehat{\zeta}$ has an ordinary flow (it is bounded) 
$\Phi: (M\times \mathbb{R}) \times \mathbb{R} \longrightarrow M \times 	\mathbb{R}$ such that 
$$\dfrac{d}{ds}\Phi_s ((x,t))=\widehat{\zeta}(\Phi_s(x,t)), \quad \Phi_0=\operatorname{Id}_{M \times \mathbb{R}}.$$
Since the second component of $\widehat{\zeta}$ is $\partial_t$, it follows that $\Phi_s$ has to be in the form $\Phi_s(x,t)=(\Phi^M(x,t,s),s+t)$ for a smooth family of maps $\Phi^M:(M\times \mathbb{R})\times \mathbb{R} \to M$.
In particular the restriction of the flow to $[0,1]$ preserves $\widehat{M}$.
Therefore defining
$\rho_t:=\Phi^M(\cdot,0,t)$ for $t\in [0,1]$ we get the flow of $\zeta_t$. In other words
the flow is defined by the commutative diagram
$$\begin{tikzcd}
	{\widehat{M}} & {\widehat{M}}  \\
	M & M
	\arrow["{\Phi_t}", from=1-1, to=1-2]
	\arrow["{\pi_M}", from=1-2, to=2-2]
	\arrow["{e_0}", from=2-1, to=1-1]
	\arrow["{\rho_t}"', from=2-1, to=2-2]
\end{tikzcd}$$
for $t\in [0,1].$
\subsection{Thom isotopy Lemma}This section is a brief reminder on families of submanifolds and the Thom isotopy Lemma. The material is standard and we are following the lecture notes \cite{larry}.

\begin{definition}\label{smoothfamily}
	A family $\{S_t\}_{t\in [0,1]}$ of submanifolds of a manifold $M$ is smooth when we can find a smooth map $F:M\times [0,1]\to L$ with $L$ a smooth manifold and a submanifold $A \hookrightarrow L$ having the following properties.
	\begin{enumerate}
		\item For every $t\in [0,1]$ the map $f_t:=F(\cdot,t):M \to L$ is transverse to $A$. This means that
		$$df_t(T_xM)+T_{f_t(x)}A=T_{f_t(x)}L, \,\,\,\textrm{at every point} \, x\in M\, \textrm{such that }f_t(x)\in A.$$
		\item $S_t=f_t^{-1}(A)$ for every $t\in [0,1]$.
	\end{enumerate}
	\end{definition}
The transversality assumption in $(1)$ implies that every $f_t^{-1}(A)$ is a smooth submanifold.
\begin{lemma}[Thom isotopy Lemma - basic version]\label{Basicthom}
	Let $M$ be a smooth manifold equipped with a smooth function $F:M \times [0,1] \to \mathbb{R}$ such that:
	\begin{enumerate}
		\item for every $t\in [0,1]$, zero is a regular value of the function $f_t:=F(\cdot,t):M\to \mathbb{R}$.
		\item Every submanifold $S_t=f_t^{-1}(0)$ is compact.
	\end{enumerate}
	Then there exists an isotopy $\rho_t:M \to M$, $t\in [0,1]$ such that for every $t\in [0,1]$:
	$$\rho_t (S_0)=S_t.$$
	The same is true for the manifolds with boundary $f_{t}^{-1}([0,\infty))$.
		\end{lemma}
\begin{proof}
We don't give the full proof but a somewhat	 detailed sketch paying attention to the parts that will be used later.
Define $$\widehat{Z}:=F^{-1}(0)\subset \widehat{M}, $$ then $S_t=f_t^{-1}(0) \cong \widehat{Z}\cap (M\times \{t\}).$ Notice that $\widehat{Z}$ is compact. 
We prove the thesis by constructing a time dependent vector field $\zeta_t$ on $M$ such that the corresponding autonomous vector field 
$\widehat{\zeta}$ on $\widehat{M}$ is tangent to $\widehat{Z}$. Indeed in this case, the flow $\Phi_t$ of $\widehat{\zeta}$ preserves $\widehat{Z}$ while mapping the slice $S_0 \cong \widehat{Z} \cap (M\times \{0\}) $ to the $t$-slice $S_t \cong \widehat{Z} \cap (M\times \{t\}) $. Therefore $$\rho_t=\pi_M\, \circ \Phi_t\, \circ e_0,$$ as defined before is the desired isotopy (with the inverse diffeomorphism provided by the opposite vector field $-\widehat{\zeta}$). 

We are left with the construction of $\zeta_t$; the condition ensuing that $\widehat{\zeta}$ is tangent to $\widehat{Z}$ is equivalent to the equation: 
\begin{equation}\label{vfield}
\dfrac{\partial F}{\partial x}(x,t) \zeta_t(x) + \dfrac{\partial F}{\partial t}(x,t)=0,
\end{equation}
 at every point $(x,t)\in \widehat{Z}$. Since the partial differential $\dfrac{\partial F}{\partial x}(m,t) $ is surjective at every point $(x_0,t_0) \in \widehat{Z}$, this equation can be locally solved. Every point $z_0=(x_0,t_0)\in \widehat{Z}$ has a neighborhood $U_{z_0}$ and a vector field $v_{z_0} \in \Gamma(U_{z_0},\pi_M^*(TM))$ such that \eqref{vfield} holds in $U_{z_0}$. 
 
 Finally the compactness of $\widehat{Z}$ implies that we find a finite cover  $\{ U_{z_1},...,U_{z_n}\}$
 of $\widehat{Z}$ 
  with corresponding vector fields $v_{z_1},...,v_{z_n}$. Put $A:=\widehat{M}\setminus \widehat{Z}$, 
  then $\{A,U_{z_1},...,U_{z_n}\}$ is a cover of the whole $\widehat{M}$.
Let $\{\rho_A,\rho_{z_1},...,\rho_{z_n}\}$ be a subordinated partition of unity; one easily checks that the vector field 
$$\zeta:=\sum_{j=1}^n \rho_{z_j} v_{z_j}$$ is a well defined 
time dependent vector field on the whole of $M$ (i.e. is a vector field on $\widehat{M}$ without the $\partial_t$ - component) 
 and has the desired property that the associated vector field $\widehat{\zeta}$ stays tangent to $\widehat{Z}$.
\end{proof}
\begin{remark}
We have presented the result in the case of a family of hypersurfaces i.e. the map $F$ is real valued. The conclusion remains valid, with the same proof for a smooth family in the sense of Definition \ref{smoothfamily}.
\end{remark}

  \section{The Moser trick}
   We follow the book \cite{zbMATH05286973,zbMATH06638013} for basic facts in symplectic geometry.
 \noindent A symplectic manifold is a couple $(M,\omega)$ where $M$ is a smooth manifold and $\omega \in \Omega^2(M)$ is a closed, non-degenerate $2$-form. Non-degenerateness means that there is and induced isomorphism of vector bundles $TM \longrightarrow T^*M$ by $v \mapsto \omega(v,\cdot)$. In particular $M$ is even dimensional: $\operatorname{dim}M=2k$ and oriented (by $\omega$) because ${\omega^k}/{(k!)}$ is a volume form. 
  
\begin{example}
  A standard example is the cotangent bundle $T^*M$ of every manifold. In this case, the symplectic structure is exact: $\omega=-d\alpha$ where $\alpha \in \Omega(T^*M)$ is a specific $1$-form, the canonical {\emph{Liouville form}} 
  $$\alpha(v)=\varphi(d_{\varphi}\pi(v)),\quad \quad v \in T_{\varphi}(T^*M), \quad \varphi \in T^*M,
$$
  with $\pi:T^*M \to M$ the projection. In local coordinates $(x_1,...,x_n,\xi_1,...,\xi_n):T^*U \to \mathbb{R}^{2n}$ induced by coordinates $(x_1,...,x_n):U \to \mathbb{R}^n$ on the base, we have:
  $\alpha=\sum_i\xi_i dx_i$ so that 
  \begin{equation}\label{Darboux}
  \omega=dx_i \wedge d\xi_i.
  \end{equation}
  \end{example}
  \medskip

The previous example is in some {\em{local sense}} universal. Indeed by the Darboux Theorem \cite[Theorem 8.1]{zbMATH05286973}, every symplectic manifold is locally 
equivalent to the cotangent bundle of an open set in the euclidean space. The equivalence here is given by symplectomorphisms. 
A symplectomorphism $\varphi:(M_1,\omega_1) \longrightarrow (M_2,\omega_2)$ between two symplectic manifolds 
is a diffeomorphism $\varphi:M_1 \to M_2$ such that $\varphi^*(\omega_2)=\omega_1$. A standard way to produce symplectomorphisms is via symplectic or Hamiltonian isotopies. 
 
Now we recall the Moser trick.
This admits several versions and can be regarded more as a method than a single result. 
Here we follow closely \cite{zbMATH06638013} presenting the general procedure in the form of a lemma. In 
Section \ref{HamThom} we will prove the parameter dependent relative form.

\begin{lemma}[Moser trick]
Let 
$M$ be a manifold with a family
$\omega_t \in \Omega^2(M)$, of symplectic forms, defined for $t\in [0,1]$ such that
$$
\dfrac{d}{dt}\omega_t=d\sigma_t, 
$$ 
for every $t\in [0,1]$, where $\sigma_{t}$ is a smooth family of one forms. 
Then there is a unique time dependent vector field $\zeta_t$, $t\in [0,1]$ satisfying Moser equation
	\begin{equation}\label{Mosereq}
		\sigma_t + \iota(\zeta)\omega_t=0
	\end{equation} The flow $\rho_t$ of $\zeta_t$ has the property that at every time $t$ for which it exists:
	$$\rho_t^*\omega_t=\omega_0.$$
	In particular if $M$ is compact $\rho_1:(M,\omega_0) \to (M,\omega_1)$ is a symplectomorphism.
\end{lemma}
\begin{proof}
The equation \eqref{Mosereq} is solved by the nondegenerateness of any $\omega_t$. Let $\rho_t$ be the flow of the unique solution $\zeta_t$; we have
\begin{align}
	\dfrac{d}{dt}\rho_t^*\omega_t&=\rho_t^*\left(\dfrac{d}{dt}\omega_t +\iota(\zeta_t)d\omega_t+d\iota(\zeta_t)\omega_t\right)\\ &=\rho_t^*(d\sigma_t+d\iota(\zeta_t)\omega_t))\\&=d\rho^*_t(\sigma_t+\iota(\zeta_t)\omega_t)=0.
\end{align}
\end{proof}
\begin{remarks}
We have presented the proof in a somewhat reversed order. Often one first looks for an isotopy such that $\rho_t^*\omega_t=\omega_0$, then by differentiation, Moser equation appears. The main assumption on the differential forms can be written as $\dfrac{d}{dt}[\omega_t]=0$ (de Rham cohomology class). In this case a smooth family of potentials $\sigma_t$ can be found. A quick proof, as indicated in \cite[Theorem 3.2.4]{zbMATH06638013}, uses basic Hodge theory. 
\end{remarks}
The next results are the parameter versions of two consequences of the Moser method. Again we follow closely \cite{zbMATH06638013} checking that all their proofs remain valid with an extra parameter.

\noindent We begin with the trivial observation that the Moser trick can be performed with a parameter. Let $M$ be a (non necessarily compact) manifold and $\omega_t^s$ for $(t,s)\in [0,1]\times[0,1]$ a family of symplectic forms with exact $t$-derivative: 
\begin{equation}\label{exactderivative}
\dfrac{d}{dt}\omega_t^s=d\sigma^s_t,\quad t,s \in [0,1],\end{equation}
for a smooth family $(\sigma^s_t)_{s,t}$ of time dependent $1$-forms. Non degenerateness implies that there is a unique $s$-family of time dependent vector fields solving Moser equation:$$
\iota(\zeta_t^s)\omega_t^s + \sigma_t^s=0.$$
If $M$ is compact we can find a smooth, two parameters family of diffeomorphisms $\psi_t^s$ such that 
$$\psi_0^s = \operatorname{Id}, \quad (\psi_t^s)^* \omega_0^s=\omega_t^s, \quad t,s \in [0,1].$$
In general, the result holds for all the values of $s$ such that the time dependent vector field $\zeta(t,s)$ admits a (local) flow.

The next result is the parametric version of \cite[Lemma 3.2.1]{zbMATH06638013}.
\begin{proposition}[Moser isotopy with parameter, \cite{zbMATH06638013}]\label{Moserisotopy}
Let $M$ be a $2n$-dimensional manifold with a compact submanifold $X$. Assume that $(\omega_0^s)_{s\in [0,1]}, (\omega_1^s)_{s\in [0,1]}$ are two $s$-families of closed $2$-forms such that for every $s\in [0,1]:$ \begin{enumerate}
\item
$\omega_0^s\restriction{X}= \omega_1^s\restriction{X}$ as differential forms (i.e. on the whole $T_XM$)
\item $\omega_0^s\restriction{X}$ and $\omega_1^s\restriction{X}$ are non degenerate on $T_XM$.
\end{enumerate}
Then there exists a neighborhood $U_0$ of $X$, a family of open neighborhoods $\left\{U_s\right\}_{s\in[0,1]}$ of $X$ and a family of diffeomorphisms $\psi^s: U_0 \longrightarrow U_s=\psi^s(U_0)$ such that  
$$\psi^s|_{X}=\operatorname{Id}_X, \quad \omega_1^s=(\psi^s)^*\omega^s_0, \quad s \in [0,1].$$
Moreover choices can be made in a way that $d\psi^s\restriction{X}=\operatorname{Id}_{T_XM}$ for every $s\in [0,1]$.
\end{proposition}
\begin{proof}
The result follows from Moser method once we find a neighborhood $U_0$ of $X$ and a smooth family $(\sigma^s)_{s\in [0,1]}$ of $1$-forms 
on $U_0$
such that
$$\omega^s_1-\omega_0^s=d\sigma^s \,\textrm{on}\,\,U_0\quad \, \textrm{and} \quad \, \sigma^s\restriction{X}=0 \,\,\, \textrm{on}\,\,T_XM,$$ for every $s \in [0,1]$. 

This granted, the family
$$\omega_t^s:=\omega_0^s+td\sigma^s$$  is symplectic on $X$ for every $s,t$. It remains symplectic, for every $s,t$ on an open set $U_0\supset X$ where
 we can find a vector field $\zeta_t^s$ such that $
\iota(\zeta_t^s)\omega_t^s + \sigma^s=0$ for $t,s \in [0,1]$. Moreover the crucial property: $$V_t^s |_{X}=0$$ for every
$t,s$ holds. Shrinking $U_0$ if required, we can apply Moser method (smooth in $s$) around $X$. 

Finally the construction of $\sigma^s$ is performed exactly in the same way as in loc. cit. using the homotopy operator on a tubular neighborhood. Since the difference $\omega_1^s-\omega_0^s$ vanishes identically on $X$, we can find $\sigma^s$ such that all its coefficients with their first partial derivatives vanish on $X$ on every coordinate chart (see \cite[Appendix 1.]{zbMATH04048539}). Such a property implies $d\psi^s\restriction{X}=\operatorname{Id}_{T_XM}$.
\end{proof}
Now comes the Weinstein symplectic neighborhood Theorem, the parametric version of \cite[Theorem 3.4.10]{zbMATH06638013}. 
\begin{proposition}[Weinstein symplectic neighborhood Theorem with parameter]\label{weinstein}
For $j=0,1$ let $(M_i,\omega_i)$ symplectic manifolds with compact symplectic submanifolds $S_i$. 
 Let $\Phi^s:\nu_{S_0} \longrightarrow \nu_{S_1}$, $s\in [0,1]$ be a family of isomorphisms of the corresponding symplectic normal bundles covering a family of symplectomorphisms $\phi^s:(S_0,\omega_0) \longrightarrow (S_1,\omega_1)$. There is a neighborhood $U_0$ of $S_0$ and a family of symplectomorphisms
 $$\psi^s:U_0 \longrightarrow \psi^s(U_0)$$ such that: $\psi^s(U_0)$ is an open neighborhood of $S_1$, 
$\psi^s$ extends $\phi^s$ and $d\psi^s|_{\nu_{S_0}}=\Phi^s$.
\end{proposition}
\begin{proof}
Choose Riemannian metrics and open neighborhoods $V_j$ of $S_j$ such that the exponential maps induce diffeomorphisms: $\operatorname{exp}_j: \nu_{S_j} \longrightarrow V_j$. Let $\Psi^s$ be the family of diffeos defined by the diagram:
$$\begin{tikzcd}
	{\nu_{S_0}} & {V_0} \\
	{\nu_{S_1}} & {V_1}
	\arrow["{\operatorname{exp}_0}", from=1-1, to=1-2]
	\arrow["{\Phi^s}"', from=1-1, to=2-1]
	\arrow["{\Psi^s}", from=1-2, to=2-2]
	\arrow["{\operatorname{exp}_1}"', from=2-1, to=2-2]
\end{tikzcd}$$
Then the $\Psi^s$ constitute a family of diffeomorphisms that extend $\phi^s$ and induce $\Phi^s$ (via the normal derivative). In this way, we may apply the Moser Isotopy Theorem, Proposition \ref{Moserisotopy} to $V_0$ equipped with the forms $\omega_0$ and $(\Psi^s)^*\omega_1$. Such forms are equal on the whole $T_{S_0}M_0$.
Of course to preserve the normal data the diffeomorphism provided by the Moser isotopy should be chosen with differential equal to the identity on the submanifold.
\end{proof}



%

\section{Hamiltonian Thom isotopy} \label{HamThom}

\noindent Let $(M,\omega)$ be a symplectic manifold. For every function $f\in C^{\infty}(M)$, the Hamiltonian vector field associated to $f$ is defined uniquely by the equation:
$$\iota(V_f)\omega=df.$$
The flow of an Hamiltonian vector field preserves the symplectic structure. A vector field is said Hamiltonian if arises from a function in this way.

Let now $\rho:M\times [0,1]\to M$ be an isotopy; if every $\rho_t$ is a symplectomorphism then $\rho$ is called a {\em{symplectic isotopy}}. A 
special class of symplectic isotopies are the {\em{Hamiltonian isotopies}},
the ones with $\zeta_t$ Hamiltonian for every $t$.
 In this case a smooth time dependent Hamiltonian $H:M\times [0,1]\longrightarrow \mathbb{R}$ such that $\iota({\zeta_t})\omega =dH(\cdot,t)$
 can be found.
 Of course if $H^1(M,\mathbb{R})=0$, every symplectic isotopy is Hamiltonian.
 
\noindent Here we prove Theorem \ref{symsubmanifolds}. We state it again for ease of reading.
  \begin{theorem}
	Let $(M,\omega)$ be a compact symplectic manifold with a smooth family $\{S_t\}_{t\in [0,1]}$ of closed symplectic submanifolds. Then there is a smooth Hamiltonian isotopy $(\rho_t)_{t\in[0,1]}$ such that $\rho_t(S_0)=S_t$, $t\in[0,1]$.
\end{theorem}
\noindent As said in the introduction this is part of the lore of the symplectic community. 
Let's list shortly the main references.
\begin{itemize}
\item[-] C. Da Silva \cite{zbMATH05286973} states the result without the Hamiltonian property of the isotopy i.e. $\rho_t$ is a symplectic isotopy.
\item[-] Auroux \cite[Proposition 4]{zbMATH01123717} shows that there is a {\em{continuous}} family of symplectomorphisms $(\rho_t)_t$ such that $\rho_0=\operatorname{Id}$ and $\rho_t(S_0)=S_t$, $t\in [0,1]$. By checking that the Moser method and its consequences, as the Weinstein symplectic neighborhood, can be performed with respect to a parameter, the same proof shows that the family of symplectomorphisms is smooth. (Actually it seems that this proof requires such smoothness.)
\item[-] Siebert and Tian \cite[Proposition 0.3]{sieberttian} prove Theorem  \ref{symsubmanifolds} in dimension $4$. Their proof uses complex coordinates, a fact that can be replaced by the existence, a priori, of a symplectic isotopy $\rho_t$ mapping $S_0$ to $S_t$.
 \end{itemize}

\noindent
We will prove Theorem \ref{symsubmanifolds} with the following steps:
\begin{enumerate}
\item we first show that the  proof of Auroux includes smooth dependence in time; 
\item adapt Siebert--Tian proof given a smooth symplectic isotopy $\rho_t$ for granted.
\end{enumerate}
\noindent We are ready to add smoothness to the proof of Auroux \cite[Proposition 4]{zbMATH01123717}.

\begin{proposition}\label{isosymplectic}
	Let $(M,\omega)$ a compact symplectic manifold with a smooth family $\{S_t\}_{t\in [0,1]}$ of closed symplectic submanifolds. 
	Then there is a smooth symplectic isotopy $(\rho_t)_{t\in[0,1]}$ such that $\rho_t(S_0)=S_t$, $t\in[0,1]$.
\end{proposition}

\begin{proof}
	Denote with $\iota_t:S_t \hookrightarrow M$ the inclusions. 	
	By  the Thom isotopy Lemma there is a smooth isotopy 
	$\varphi: M\times [0,1] \longrightarrow M$ with $\varphi_t(S_0)=S_t$ that combined with
	Moser stability (with smooth parameter) produces a smooth family of symplectomorphisms
	$$\psi^t:(S_0,\iota_0^*\omega) \longrightarrow (S_t,\iota_t^*\omega).$$
	Such a family is covered by a 
 smooth family  $N^{\omega}S_0 \longrightarrow N^{\omega}S_t$ of isomorphisms between the
	symplectic normal bundles of the submanifolds $S_t$. 
	Indeed one first considers the principal bundle $P \longrightarrow S_0 \times [0,1]$ of the normal symplectic frames of the submanifolds; the structure group being the symplectic group. 
	Parallel transport with respect to a connection on $P$ gives a family of isomorphisms of principal bundles $P\rest{S_0}\longrightarrow P\rest{S_t}$. This induces isomorphisms on all the associated bundles such as the symplectic normal bundles. By proposition \ref{weinstein} we end up with a smooth family $\psi^t: U_0 \longrightarrow U_t$ of symplectomorphisms between tubular neighborhoods $U_t \supset S_t$. 
	Let now $\mu_t:M \longrightarrow M$ be any smooth family of diffeomorphisms of the ambient extending $\psi^t$ and, following closely \cite{zbMATH01123717}, put:
	 $$\omega_t:=\mu_t^*\omega, \quad \Omega_t=-\dfrac{d}{dt}\omega_t.$$
	 Assume we can find a vector field $\zeta_t$ such that:
	 \begin{itemize}
	 \item[-] the forms $\alpha_t:=\iota(\zeta_t)\omega_t$ satisfy $d\alpha_t=\Omega_t$,  
	 \item[-] $\zeta_t \rest{S_0}$ is tangent to $S_0$ for every $t$.
	 \end{itemize}
Then the proof is completed for if $\sigma_t$ denotes the flow of $\zeta_t$, let $\rho_t:= \mu_t \circ \sigma_t$; we have $\mathcal{L}_{\zeta_t}\omega_t=\Omega_t$ by the Cartan formula so that:
$$\dfrac{d}{dt}\rho_t^*\omega=
\dfrac{d}{dt}(\sigma_t^*\omega_t)=
(\sigma_t)^*\left( \dfrac{d}{dt}\omega_t+\mathcal{L}_{\zeta_t}\omega_t\right)=0.$$
This means that $\rho_t$ is a family of symplectomorphisms mapping $S_0$ to $S_t$.	

Let us show how to find $\zeta_t$ or equivalently $\alpha_t$. By construction $\omega_t$ is constant on $U_0$ so that the condition: $\zeta_t\rest{S_0}\in \Gamma(TS_0)$ means exactly that $N^{\omega}S_0 \subset \operatorname{Ker} \alpha_t$. Now all the $\omega_t$ are cohomologous, which implies $[\Omega_t]=0$ in $H^2(M,\mathbb{R})$. A smooth family of potentials $\beta_t\in \Omega^1(M)$ on the entire $M$,
 such that $d\beta_t=\Omega_t$ can be found.
The smoothness following exactly by the argument in the proof of \cite[Theorem 3.17]{zbMATH06638013}.
On $U_0$ we have $d\beta_t=\Omega_t=0$ i.e. the forms $\beta_t$ define classes in $H^1(U_0,\mathbb{R})$. Now, let $\pi: U_0 \longrightarrow S_0$ 
the projection of a tubular neighborhood 
satisfying $T_x\pi^{-1}(x)=N_x^{\omega}S_0$ at every point $x\in S_0$ and define: $$\gamma_t:=(\iota_0 \circ \pi)^*\beta_t\in \Omega^1(U_0)$$ with $\iota_0:S_0 \hookrightarrow M$ the inclusion. Since $(\iota_0 \circ \pi)^*$ is the identity in cohomology we have $[\gamma_t]=[\beta_t \rest{U_0}]$ in $H^1(U_0,\mathbb{R})$. It follows that there is a smooth family of functions $f_t\in C^{\infty}(U_0,\mathbb{R})$ with $\gamma_t=\beta_t+df_t$ in $U_0$. Let $g_t$ be any family of functions on the whole $M$ extending $f_t$ and put $\alpha_t:=\beta_t+dg_t$. We have found our $\alpha_t$. Indeed $d\alpha_t=d\beta_t=\Omega_t$ and from $\alpha_t \restriction{U_0}=\gamma_t$ it follows that $N_x^{\omega}S_0 \subset \operatorname{ker}\left( \alpha_t\restriction{x}\right)$ at every point $x$ in $S_0$.
		\end{proof}

We finally prove Theorem \ref{symsubmanifolds} following the idea of the proof of \cite[Proposition 0.3]{sieberttian}.
\begin{proof}[Proof of Theorem \ref{symsubmanifolds}]
Following the proof of lemma \ref{Basicthom}, it is enough to construct a Hamiltonian function $H:M\times I\to \R$ such that,
with $h_t:=H(\cdot,t):M \to \mathbb{R}$, the vector field 
 $$\widehat{V}_{h_t}=V_{H_t}+\partial_t$$ is tangent to $\widehat{S}=\{(x,t)\in M\times[0,1]: \, x\in S_t\}$. To this end, for every $z=(p_0,t_0)\in M\times I$ we will find a neighborhood $U_z=U_{p_0}\times U_{t_0}$ and a local Hamiltonian function $H_z: U_z\to \R$ such that, again denoting by $h_{z,t}:=H_z(\cdot, t):U_{p_0}\to \R$, two conditions are satisfied: 
\begin{enumerate}
\item $V_{h_{z,t}}+\partial_t$ is tangent to $ U_z\cap \widehat{S}$,
\item $H_z$ vanishes on $U_z\cap \widehat{S}$.
\end{enumerate}
The global Hamiltonian $H$ will then be defined using a partition of unity $\{\theta_z\}_{z\in M\times I}$ subordinated to the cover $\{U_z\}_{z\in M\times I}$, i.e. we will set $H:=\sum_{z}\theta_zH_z$. Its Hamiltonian vector field (at every fixed $t$) for such $H$ will satisfy:
 \beq \label{eq:sumrho} \iota({V_{h_t}})\omega=\sum_{z}\dd \theta_z h_{z,t}+\theta_z\dd h_{z,t}.\eeq
 Condition (2) ensures that along $\widehat{S}$ the first summand in \eqref{eq:sumrho} vanishes; together with condition (1), this ensures that $V_{h_t}$ is tangent to $\widehat{S}$.
 
 We proceed now with the construction of such  functions $H_z$
using a  symplectic isotopy $\rho_t:M \longrightarrow M$ with $t\in [0,1]$ such that $\rho_t(S_t)=S_0$. This is (or better its inverse) provided by  proposition \ref{isosymplectic}. 
  
 For every point $z=(p_0, t_0)$ in the open set $(M\times [0,1])\setminus \widehat{S}$ we pick a small neighborhood $U_z=U_{p_0}\times U_{t_0}$ disjoint from $\widehat{S}$ and  we set $H_z\equiv0$.
 
 For every point $z=(p_0, 0)\in \widehat{S}$ we proceed as follows. Since $S_{0}$ is symplectic, we can find a neighborhood $U_{p_0}\subset M$ and coordinates 
 $\Theta:U_{p_0}\longrightarrow \R^{n}_x \times \mathbb{R}^n_y$ such that $U_{p_0}\cap S_{0}$ is described by $\left\{x_j=y_j=0,\, j=m+1,...,n\right\}$ and
 \beq\label{eq:symplcoord} \omega=\sum_{i=1}^{n}\dd x_i\wedge \dd y_i\quad \textrm{with}\quad e_0^* (\omega)=\sum_{i=1}^{m}\dd x_i\wedge \dd y_i.\eeq
Indeed this is (again) consequence of Weinstein symplectic neighborhood Theorem, or better it's non compact version  \cite[Theorem 15.2]{zbMATH04048539}. One first locally trivializes the symplectic normal bundle, then applies the symplectic neighborhood Theorem to split the symplectic structure as a product,
 then applies the standard Darboux Theorem on the two factors.
 In alternative, the Carath\'eodory--Jacobi--Lie Theorem \cite[Theorem 17.2]{zbMATH04048539} can be used.
 
 The above coordinate system induces coordinates in $\widehat{M}$ of the form:
 $$\widehat{\Theta}:{U}_z:=\left\{(p,t): \rho_t(p)\in U_{p_0} \right\} \longrightarrow \mathbb{R}^n_{\widehat{x}} \times \mathbb{R}^n_{\widehat{y}} \times [0,1]_{\tau}$$ with:
 $$\widehat{x}_i(p,t):=x_i(\rho_t(p)), \quad \widehat{y}_i(p,t):=y_i(\rho_t(p)), \quad \tau(p,t)=t.$$
In such coordinates we have: $\widehat{S}\cap U_{z}= \left\{ \widehat{x}_i=0=\widehat{y}_i ,\, \, i=m+1,...,0\right\}$ so that $\partial_{\tau}$ is tangent to $\widehat{S}$.
 Let $\zeta_t$ be the vector field associated to $\rho_t$, i.e: $\dfrac{d}{dt}\rho_t ={\zeta}_t \circ \rho_t$ for $t\in[0,1]$, with coordinate representation 
 $\sum_{j=1}^n \xi_t^j \partial_{x_j}+ \eta_t^j\partial_{y_j}$ in $U_{p_0}$. 
   $$\xi_t^j(p)=\dfrac{d}{ds}\rest{s=t}x_j \big{(} \rho_s(\rho_t^{-1}(p))\big{)}, \quad \eta_t^j(p)=\dfrac{d}{ds}\rest{s=t}y_j \big{(} \rho_s(\rho_t^{-1}(p))\big{)}.$$
Using the inclusions  $e_t:M \hookrightarrow M \times\{t\}\subset \widehat{M}$ at time $t$,   locally defined vector fields on $M$ can be considered on $\widehat{M}$ omitting (when the context is clear) further notations. In this sense:
\begin{equation}\label{coordinatefields}
d\left(\rho_t^{-1}\right)  \partial_{x_i}\rest{p}=\partial_{\,\widehat{x_i}}\rest{(p,t)}
\end{equation}
 and similar identities for $\widehat{y_j}$.
Since the vector field $\partial_t$ on $\widehat{M}$ is given by
$\partial_t \rest{(p,t)}=\dfrac{d}{ds}\rest{s=t} \gamma_p(s)$ with $\gamma_p(s)=(p,s)$, 
we have, over $U_z$ the formula:
  \begin{equation}\label{timetvectorfield}
  \partial_t\rest{(p,t)}=\partial_{\tau}+\sum_{j=1}^n \xi_t^j(\rho_t(p))\partial_{\,\widehat{x}_j}\rest{(p,t)}+\eta_t^j(\rho_t(p))\partial_{\widehat{y}_j}\rest{(p,t)}.\end{equation}
  We define, in $U_z$ the local Hamiltonian by:
 \beq 
 H_{z}(p,t)=\sum_{j=m+1}^n\eta^j_t(\rho_t(p))\,\widehat{x}_j(p,t)- \xi_t^j(\rho_t(p))\,\widehat{y}_j(p,t).
 \eeq 
 Clearly  $H_z$ vanishes on $U_z\cap \widehat{S},$ i.e. it satisfies condition (2) above. As for condition (1), since every $\rho_t$ is a symplectomorphism the Hamiltonian vector fields are $\rho_t$-related:
 $$V_{h_t} \rest p =d_q\left(\rho_t^{-1}\right) \left(V_{k_t} \rest q\right), \quad q=\rho_t(p),$$
 where $k_t$ is the local Hamiltonian function $k_t:U_{p_0}\longrightarrow \mathbb{R}$ defined by $k_t=\sum_{j=m+1}^n \eta_j^t\, x_j - \xi_j^t \,y_j$. For the Hamiltonian vector field of $k_t$ we compute
 $$V_{k_t} \rest{S_0}=-\sum_{j=m+1}^n
\left( \xi_j^t\, \partial_{x_j}+
  \eta_j^t\, \partial_{y_j}\right) \rest{S_0}.$$
It follows, using formulas \eqref{coordinatefields} and \eqref{timetvectorfield}  that at every point $(p,t)\in \widehat{S}\cap U_z$ we have:  
$$\left(V_{h_t}+ {\partial}_{\tau}\right)\rest{(p,t)}=\partial_{\tau} \rest{(p,t)}+ \sum_{j=1}^m \xi_t^j(\rho_t(p))\partial_{\,\widehat{x}_j}\rest{(p,t)}+\eta_t^j(\rho_t(p))\partial_{\widehat{y}_j}\rest{(p,t)},$$ 
 which is manifestly tangent to $\widehat{S}$.
 The proof is completed because the open sets $U_z$ in the form discussed above constitute an open covering of $\widehat{S}$.
  \end{proof}

\addtocontents{toc}{\protect\setcounter{tocdepth}{2}}

\bibliographystyle{acm}
\bibliography{literature.bib}

\end{document}